\pdfoutput=1
\documentclass[12pt]{article}

\usepackage{lineno}

\usepackage[usenames]{color}
\usepackage[colorlinks=true,
linkcolor=webgreen,
filecolor=webbrown,
citecolor=webgreen]{hyperref}

\definecolor{webgreen}{rgb}{0,.5,0}
\definecolor{webbrown}{rgb}{.6,0,0}

\newcommand{\seqnum}[1]{\href{https://oeis.org/#1}{\rm \underline{#1}}}

\usepackage{amsmath, amssymb, amscd, amsthm, amsfonts}
\usepackage{mathtools}
\usepackage{tabto}
\usepackage{tabularx}
\usepackage[makeroom]{cancel}
\usepackage{fullpage}
\usepackage{float}
\usepackage{longtable}

\usepackage[tableposition=below]{caption}
\newtheorem{theorem}{Theorem}[section]
\newtheorem{lemma}[theorem]{Lemma}

\newtheorem{corollary}[theorem]{Corollary}

\newtheorem{definition}{Definition}

\newcommand{\INFIX}{\geq_{\rm inf}}
\newcommand{\SUFFIX}{\geq_{\rm suff}}
\newcommand{\PREFIX}{\geq_{\rm pref}}
\newcommand{\VMAT}{\begin{bmatrix}
	1 & 0
\end{bmatrix}}
\newcommand{\WMAT}{\begin{bmatrix}
	1 \\ 0
\end{bmatrix}
}
\newcommand{\ZMAT}{\begin{bmatrix}
	0 & 0 \\
	1 & 0
\end{bmatrix}
}
\newcommand{\IMAT}{\begin{bmatrix}
	1 & 0 \\
	0 & 1
\end{bmatrix}
}

\definecolor{green}{RGB}{0,127,0}
\definecolor{red}{RGB}{200,0,0}

\begin{document}

		\title{Record-Setters in the Stern Sequence}
		
		%% or include affiliations in footnotes:
		\author{Ali Keramatipour\\
		School of Electrical and Computer Engineering\\
		University of Tehran\\
		Tehran\\
		Iran\\
		\href{mailto:alikeramatipour@ut.ac.ir}{\tt alikeramatipour@ut.ac.ir} \\
		\and
		Jeffrey Shallit\\
		School of Computer Science\\
University of Waterloo\\
Waterloo, ON  N2L 3G1 \\
Canada\\
\href{mailto:shallit@uwaterloo.ca}{\tt shallit@uwaterloo.ca}}
		
		\maketitle
		
		\begin{abstract}
			Stern's diatomic series, denoted by $(a(n))_{n \geq 0}$, is defined by the
			recurrence relations 
			$a(2n) = a(n)$ and $a(2n + 1) = a(n) + a(n + 1)$ for $n \geq 1$,
			and initial values $a(0) = 0$ and $a(1) = 1$.
			A record-setter for a sequence $(s(n))_{n \geq 0}$ is an index $v$ such that $s(i) < s(v)$ holds for all $i < v$.  In this
			paper, we give a complete description of the record-setters for
			the Stern sequence.
		\end{abstract}

	\section{Introduction}\label{section-introduction}
	
	Stern's sequence $(a(n))_{n \geq 0}$, defined by the recurrence relations
	$$ a(2n) = a(n), \quad a(2n+1) = a(n)+a(n+1),$$
	for $n \geq 0$, and initial values $a(0) = 0$, $a(1) = 1$,
	has been studied for over 150 years.  It was introduced by Stern in 1858 \cite{Stern:1858}, and later studied by
	Lucas \cite{Lucas:1878}, Lehmer \cite{Lehmer:1929}, and many others.  For a survey of the Stern sequence and its amazing properties, see the papers of Urbiha \cite{Urbiha:2001} and Northshield
	\cite{Northshield:2010}.   It is an example of a $2$-regular sequence \cite[Example 7]{Allouche&Shallit:1992}.  The first few values of this sequence are
	given in Table~\ref{tab1}; it is sequence \seqnum{A002487} in
	the {\it On-Line Encyclopedia of Integer Sequences} (OEIS)\cite{Sloane:2022}.
	\begin{table}[H]
	\begin{center}
	\begin{tabular}{c|cccccccccccccccc}
		$n$ 		& 0 & 1 & 2 & 3 & 4 & 5 & 6 & 7 & 8 & 9 & 10 & 11 & 12 & 13 & 14 & 15\\
		\hline
				$a(n)$ 	& 0 & 1 & 1 & 2 & 1 & 3 & 2 & 3 & 1 & 4 & 3 & 5  & 2  & 5  & 3  & 4  
\end{tabular}
\end{center}
\caption{First few values of the Stern sequence.}
\label{tab1}
	\end{table}
	
	The sequence $a(n)$ rises and falls in a rather complicated way; see
	Figure~\ref{fig1}.
	\begin{figure}[htb]
	\begin{center}
	    \includegraphics[width=6.5in]{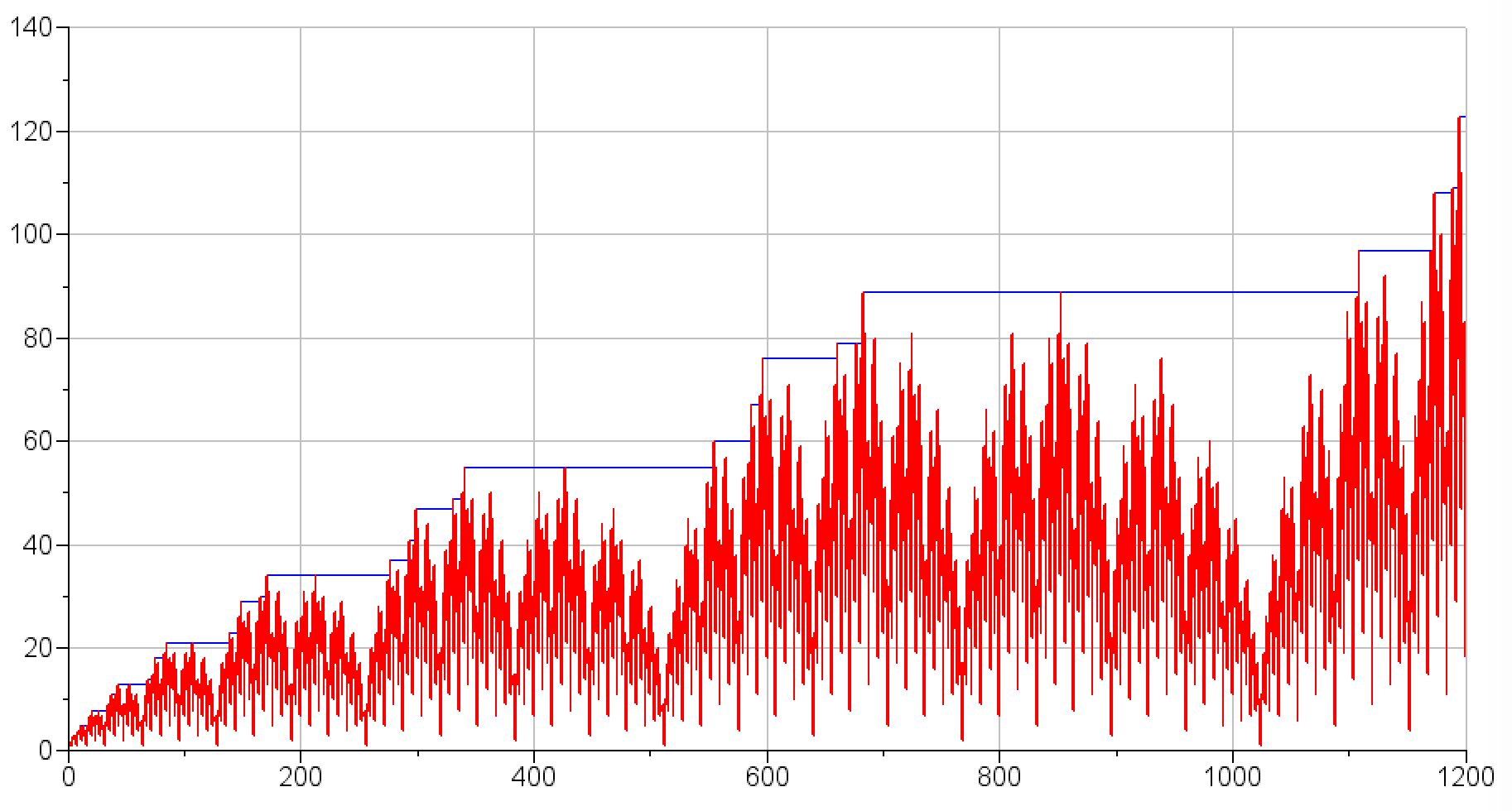}
	\end{center}
	\caption{Stern's sequence and its running maximum for $0\leq n \leq 1200$.}
	\label{fig1}
	\end{figure}
	For this reason, several authors have been interested in understanding the
	local maxima of $(a(n))_{n \geq 0}$.   
	This is easiest to determine when one restricts one's attention to numbers with $i$ bits; that is, to the interval $[2^{i-1}, 2^{i})$.  Lucas \cite{Lucas:1878} observed without proof
	that $\max_{2^{i-1} \leq n < 2^i} a(n) = F_{i+1}$, where
	$F_n$ is the $n$th Fibonacci number, defined as usual by
	$F_0 = 0$, $F_1 = 1$, and $F_n = F_{n-1} + F_{n-2}$ for $n \geq 2$, and proofs
	were later supplied by Lehmer \cite{Lehmer:1929} and Lind \cite{Lind:1969}.
	The second- and third-largest values in the same interval,
	$[2^{i-1}, 2^{i})$, were determined by Lansing \cite{Lansing:2014}, and
	more general results for these intervals were obtained by Paulin \cite{Paulin:2017}.
	
	On the other hand,
	Coons and Tyler \cite{Coons&Tyler:2014} showed that
	$$ \limsup_{n \rightarrow \infty} \frac{a(n)}{n^{\log_2 \varphi}} = 
	    \frac{\varphi^{\log_2 3}}{\sqrt{5}},$$
	  where $\varphi = (1+\sqrt{5})/2$ is the golden ratio.  This gives
	  the maximum order of growth of Stern's sequence.  Later, Defant \cite{Defant:2016} generalized their result to the analogue of Stern's sequence in all integer bases $b \geq 2$.

	In this paper, we are concerned with the positions of the ``running maxima'' or
	``record-setters'' of the Stern sequence overall, not restricted to
	subintervals of the form $[2^{i-1}, 2^i)$.  These are the indices $v$
	such that $a(j) < a(v)$ for all $j < v$.   The first few record-setters
	and their values are given in Table~\ref{tab2}.
		\begin{table}[H]
	\begin{center}
	\begin{tabular}{c|cccccccccccccccccc}
		$i$ 		& 0 & 1 & 2 & 3 & 4 & 5 & 6 & 7 & 8 & 9 & 10 & 11 & 12 & 13 & 14 & 15 & 16 & 17 \\
		\hline
				$v_i$ 	& 0 & 1 & 3 & 5 & 9 & 11 & 19 & 21 & 35 & 37 & 43 & 69&  73 & 75 & 83 & 85 & 139 & 147  \\
				$a(v_i)$ & 0 & 1 & 2 & 3 & 4 & 5  & 7 & 8 & 9 & 11 &  13 & 14 & 15 &  18 & 19 & 21 & 23  &26
\end{tabular}
\end{center}
\caption{First few record-setters for the Stern sequence.}
\label{tab2}
	\end{table}
The sequence of record-setters $(v_i)_{i \geq 1}$ is sequence \seqnum{A212288} in the OEIS,
and the sequence $(a(v_i))_{i \geq 1}$ is sequence \seqnum{A212289} in the OEIS.
In this paper, we provide a complete description of the record-setters for
the Stern sequence.   

To state the theorem, we need to use a standard notation for
repetitions of strings:  for a string $x$, the expression
$x^i$ means $\overbrace{xx\cdots x}^i$.   Thus, there is a possibility for confusion between ordinary powers of integers
and powers of strings, but hopefully the context will make our meaning clear.
	
\begin{theorem} \label{mainTheorem}
	The $k$-bit record-setters, for $k < 12$, are
	given in Table~\ref{tab3}.
	
	For $k \geq 12$,
		the $k$-bit record-setters of the Stern sequence, listed 
		in increasing order, have the following representation in base $2$:
    \begin{itemize}
        \item $k$ even, $k = 2n$:
            $$\begin{cases}
				100\, (10)^a\, 0\, (10)^{n-3-a}\, 11,  & \text{ for } 0 \leq a \leq n-3; \\
				 (10)^{b}\, 0\, (10)^{n-b-1} \, 1, & \text{ for } 1 \leq b \leq \lfloor n/2 \rfloor; \\
				(10)^{n-1}\, 11.
            \end{cases}$$

        \item $k$ odd, $k=2n+1$:
$$
            \begin{cases}
                10 00\, (10)^{n-2}\, 1 ; \\
                100100\, (10)^{n-4}\, 011; \\
                100\, (10)^b\, 0\, (10)^{n-2-b} \, 1, & \text{ for } 1 \leq b \leq \lceil n/2 \rceil - 1; \\
				(10)^{a+1}\,
				0\, (10)^{n-2-a}\, 11, & \text{ for } 0 \leq a \leq n-2;\\
				(10)^{n}\, 1. 
				\end{cases}
$$
	\end{itemize}
	In particular, for $k \geq 12$, the number of $k$-bit record-setters
	is $\lfloor 3k/4 \rfloor - (-1)^k$.
	\end{theorem}

	In this paper, we prove the correctness of the classification above by ruling out many cases and then trying to find the set of record-setters.

	Our approach is to interpret numbers as binary strings. In Section \ref{basics}, we will introduce and provide some basic lemmas regarding this approach.

	To find the set of record-setters, we exclude many candidates and prove they do not belong to the set of record-setters in Section \ref{search_space}.

	In Section \ref{limit1001000}, we rule out more candidates by using some calculations based on Fibonacci numbers.

Finally, in Sections \ref{final_even} and \ref{final_odd}, we finish the classification of record-setters and prove Theorem \ref{mainTheorem}.

{\small\begin{center}
		\begin{longtable}[htb]{c|r|r}
		$k$ & record-setters & numerical \\
			& with $k$ bits & values  \\
		\hline
		1 & 1 & 1 \\
		2 & 11 & 3 \\
		3 & 101 & 5 \\
		4 & 1001 & 9 \\
		  & 1011 & 11 \\
		5 &  10011  &  19 \\
		  &  10101  &  21 \\
		6 &  100011  &  35 \\
		  &  100101  &  37 \\
		  &  101011  &  43 \\
		7 &  1000101  &  69 \\
		  &  1001001  &  73 \\
		  &  1001011  &  75 \\
		  &  1010011  &  83 \\
		  &  1010101  &  85 \\
		8 &  10001011  &  139 \\
		  &  10010011  &  147 \\
		  &  10010101  &  149 \\
		  &  10100101  &  165 \\
		  &  10101011  &  171 \\
		9 &  100010101  &  277 \\
		  &  100100101  &  293 \\
		  &  100101011  &  299 \\
		  &  101001011  &  331 \\
		  &  101010011  &  339 \\
		  &  101010101  &  341 \\
		10 &  1000101011  &  555 \\
		   &  1001001011  &  587 \\
		   &  1001010011  &  595 \\
		   &  1001010101  &  597 \\
		   &  1010010101  &  661 \\
		   &  1010101011  &  683 \\
		11 &  10001010101  &  1109 \\
		   &  10010010101  &  1173 \\
		   &  10010100101  &  1189 \\
		   &  10010101011  &  1195 \\
		   &  10100101011  &  1323 \\
		   &  10101001011  &  1355 \\
		   &  10101010011  &  1363 \\
		   &  10101010101  &  1365 \\
		   \caption{$k$-bit record-setters for $k < 12$.}
		\label{tab3}
		\end{longtable}
	\end{center}
    }

	\section{Basics}\label{basics}
	We start off by defining a new sequence $(s(n))_{n \geq 0}$, which is the
	Stern sequence shifted by one: $s(n) = a(n + 1)$ for $n \geq 0$.  Henceforth we will be mainly concerned with $s$ instead of $a$.  Let $R$ be the set of record-setters
	for the sequence $(s(n))_{n \geq 0}$, so that
	$R = \{ v_i - 1 \, : \, i \geq 1 \}$.

	A {\it hyperbinary representation\/} of a positive integer $n$ is a summation of powers of $2$, using each power
	at most twice.
	The following theorem of Carlitz \cite{Carlitz:1964} provides another way of interpreting the quantity $s(n)$:
	\begin{theorem}
		The number of hyperbinary representations of $n$ is $s(n)$.
	\end{theorem}
	
	We now define some notation.  We frequently represent integers as strings of digits.
	If $ x = e_{t-1} e_{t-2} \cdots e_1 e_0$
	is a string of digits 0, 1, or 2, then $[x]_2$ denotes the
	integer $n = \sum_{0 \leq i < t} e_i 2^i$.  For example,
	\begin{equation*}
	    43 = [101011]_2 = [012211]_2 = [020211]_2 = [021011]_2 = [100211]_2.
	\label{example43}
	\end{equation*}
	By ``breaking the power $2^i$'' or the $(i + 1)$-th bit from the right-hand side, we mean writing $2^i$ as two copies of $2^{i - 1}$. For example, breaking the power $2^1$ into
	$2^0 + 2^0$ can be thought of as rewriting the string $10$ as $02$.

	Now we state two helpful but straightforward lemmas:
	
% 	??? rewrite the lemma to make clear what you are talking about is a process for converting one hyperbinary representation to another ???
	
	\begin{lemma} \label{breakBits} Let string $x$ be the binary representation of $n \geq 0$, that is $(x)_2 = n$. All proper hyperbinary representations of $n$ can be reached from $x$, only by breaking powers $2^i$, for $0 < i <|x|$.
	\end{lemma}
	\begin{proof}
	    To prove this, consider a hyperbinary representation string $y = c_{t-1} c_{t-2} \cdots c_1 c_0$ of $n$. We show that $y$ can be reached from $x$ using the following algorithm: Let $i$ be the position of $y$'s leftmost 2. In each round, change bits $c_i := c_i - 2$ and $c_{i+1} := c_{i+1} + 1$. By applying this algorithm, $i$ increases until the number of 2s decrease, while the value $[y]_2$ remains the  same. Since $i$ cannot exceed $t - 1$, eventually $y$ would have no 2s. Therefore, string $y$ becomes $x$. By reversing these steps, we can reach the initial value of $y$ from $x$, only by ``breaking" bits.
	\end{proof}

	\begin{lemma} \label{breaktwice} Let string $x$ be the binary representation of $n \geq 0$. In the process of reaching a hyperbinary representation from $x$, only by breaking bits, a bit cannot be broken twice.
	\end{lemma}
	\begin{proof}
		Since $2^i > 2^{i-1} + \cdots + 2^0$, and $[2(0)^i]_2$ $>$ $[(2)^{i-1}]_2$, the $(i+1)$-th bit from right cannot be broken twice.
	\end{proof}
	
	For simplicity, we define a new function, $G(x)$, and work with binary and hyperbinary representations henceforward.
	The argument of $G$ is a string $x$ containing only the digits $\{0,1,2, 3\}$, and its value is 
	the number of different hyperbinary representations reachable from $x$, only by the breaking mechanism we defined above. Thus, for example,
	Eq.~\eqref{example43} demonstrates that $G(101011) = 5$.
	Although the digit 3 cannot appear in a proper hyperbinary representation, we use it here to mean that the corresponding bit \textit{must} be broken. Also, from Lemma~\ref{breaktwice}, we know that the digit 4 cannot appear since it must be broken twice. We can conclude from Lemma \ref{breakBits}, for a \textit{binary} string $x$, we have $G(x) = s([x]_2)$. We define $G(\epsilon)= 1$.
	
	In what follows, all variables have the domain $\{ 0,1 \}^*$; if we have a need for the digits $2$ and $3$, we write them explicitly.
	
	We will later use the following lemma to get rid of 2s and 3s in our hyperbinary representations and get a representation using only $0$s and $1$s:
	\begin{lemma} \label{remove23} For a binary string $h$, the equalities
		\begin{itemize}
			\item[(a)] $G(2h) = G(1h)$, 
			\item[(b)] $G(30h) = G(1h)$, 
			\item[(c)] $G(3(1)^i0h) = G(1h)$,
			\item[(d)] $G(3(1)^i) = G(3) = 0$
			\end{itemize}
		hold.
	\end{lemma}
	\begin{proof}
	\leavevmode
	\begin{itemize}
	    \item[(a)]
		According to Lemma \ref{breaktwice}, we cannot break the leftmost bit twice. Therefore, the number of different hyperbinary representations we can reach from $2h$ and $1h$, i.e. their $G$-value, is the same.
		
		\item[(b)] Since 3 cannot appear in a hyperbinary representation, we must break it. This results in a new string
		$22h$. Due to Lemma \ref{breaktwice}, the first (leftmost) $2$ is useless,
		and we cannot break it again. Thus, $G(30h) = G(2h)
		= G(1h)$.
		
		\item[(c)] Since we have to break the 3 again, the string $3(1)^i0h$ becomes $23(1)^{i -1}0h$, and  $G(3(1)^i0h) = G(3(1)^{i -1}0h)$ . By continuing this we get $G(3(1)^i0h)
		= G(30h) = G(1h)$.
		
		\item[(d)] To calculate $3(1)^i$'s $G$-value, we must count the number of proper hyperbinary representations reachable from $3(1)^i$. The first 3 must be broken, and by breaking 3, we obtain another string of the same format, i.e., $3(1)^{i-1}$. By continuing this, we reach the string $3$, which cannot be broken any further and is not a valid hyperbinary string. Therefore $G(3(1)^i) = G(3) = 0$
		\end{itemize}
	\end{proof}

%??? We now define two transformations on strings.

	We now define two transformations on string $h$, prime and double prime transformations.
	For a string $h$, we let $h'$ be the string resulting from adding two to its leftmost bit, and then applying Lemma~\ref{remove23} to remove the excessively created 2 or 3. Therefore, string $h'$ is either a {\it binary} string, or it is 3, which is not transformable as the case (d) in Lemma~\ref{remove23}. For example,
	\begin{itemize}
		\item[(a)] If $h = 0011$, then we get $2011$, and by applying Lemma~\ref{remove23}, we have $h' =1011$.
		\item[(b)] If $h = 1011$, then $h' = 111$.
		\item[(c)] If $h = \epsilon$, then $h$ has no leftmost bit, and $h'$ is undefined. Therefore, we set $\epsilon' = 3$ and $G(\epsilon') = 0$.
		\item[(d)] If $h = 1$, then $h' = 3$ and $G(h') = 0$.
	\end{itemize}
	We let $h''$ be the string resulting from removing all trailing zeroes and decreasing the rightmost bit by 1. For
	example,   
	\begin{itemize}
		\item[(a)] If $h = 100\ 100$, then $h'' = 1000$;
		\item[(b)] If $h = 1011$, then $h'' = 10\ 10$;
		\item[(c)] If $h = 3$, then $h'' = 2$;
		\item[(d)] If $h = 0^i$ for
		$i \geq 0$, then after removing trailing zeros, the string does not have a rightmost bit and is not in the transformation function's domain. Therefore, we set $G(h'') = 0$.
	\end{itemize}

	The reason behind defining prime and double prime of strings is to allow dividing a single string into two pieces and calculating the $G$ function for both pieces. This way, we can calculate $G$-values more easily. For example, $h'$ is useful when a bit with the value $2^{|h|}$ is broken, and $h''$ is useful when we want to break $2^0$ and pass it to another string on its right. Lemma~\ref{breaktwice} implies this usefulness as we cannot break a bit twice; thus, we can assume the two pieces are entirely separate after breaking a bit.

	\section{Ruling out Candidates for Record-Setters}\label{search_space}
	
	%??? maybe rephrase using the term 'candidate'
	%and 'rule out candidates' for record-setters.
	
	In this section, by using Lemmas \ref{breakBits} and \ref{remove23}, we try to decrease the search space as much as possible. A useful tool is linear algebra. We now define a certain matrix $\mu(x)$ for a binary string $x$. We set
	\begin{equation}
		\mu(x) = 
		\begin{bmatrix}
			G(x) & G(x'')\\
			G(x') & G((x')'')
		\end{bmatrix} .
	\end{equation}
	
	%AAA we basically aren't setting these; this is their value
	
	For example, when $|x|=1$, the values are 
	\begin{align*}
	    &G(1) = 1, && G(1'') = G(0) = 1,\\
	    &G(1') = G(3) = 0, && G( (1')'') = G(3'') = G(2) = G(1) = 1,\\
	    &G(0) = 1, && G(0'') = 0,\\
	    &G(0') = G(2) = 1, && G( (0')'') = G(2'') = G(1) = 1,
	\end{align*}
	and the corresponding matrices are
	\begin{equation*}
	    \mu(1) = 
		\begin{bmatrix}
			1 & 1\\
			0 & 1
		\end{bmatrix}
		\text{ and }
		\mu(0) = 
		\begin{bmatrix}
			1 & 0\\
			1 & 1
		\end{bmatrix}.
	\end{equation*}
	
	In the case where $x = \epsilon$, the values are
	\begin{align*}
	    &G(\epsilon) = 1, && G(\epsilon'') = 0,\\
	    &G(\epsilon') = G(3) = 0, && G( (\epsilon')'') = G(3'') = G(2) = G(1) = 1,\\
	\end{align*}
	and the matrix is
	\begin{equation*}
		\mu(\epsilon) = 
		\begin{bmatrix}
			1 & 0\\
			0 & 1
		\end{bmatrix},
	\end{equation*}
	the identity matrix.
	
	\begin{theorem} \label{matrix_linearization}
		For two binary strings $x$ and $y$, the equation
		\begin{equation}
		    \mu(xy) = \mu(x)\cdot\mu(y)
		\end{equation}
		holds.
	\end{theorem}
	\begin{proof}
		To show this, we prove $\mu(1x) = \mu(1)\cdot\mu(x)$ and $\mu(0x) = \mu(0) \cdot \mu(x)$. The general case for $\mu(xy) = \mu(x)\cdot\mu(y)$ then follows by induction.
		
		We first prove the case for $1x$. Consider
		\begin{equation*}
			\mu(1)\cdot\mu(x) =
			\begin{bmatrix}
				1 & 1\\
				0 & 1
			\end{bmatrix} \cdot
			\begin{bmatrix}
				G(x) & G(x'')\\
				G(x') & G((x')'')
			\end{bmatrix}
			=
			\begin{bmatrix}
				G(x) + G(x') & G(x'') + G((x')'')\\
				G(x') & G((x')'')
			\end{bmatrix},
		\end{equation*}
		which must equal
		\begin{equation*}
			\mu(1x) = 
			\begin{bmatrix}
				G(1x) & G((1x)'')\\
				G((1x)') & G(((1x)')'')
			\end{bmatrix}.
		\end{equation*}

		We first prove $G(1x) = G(x) + G(x')$. Consider two cases where the first 1 either breaks or not. The number of hyperbinary representations where it does not break equals $G(x)$; if it breaks, then the rest of the string becomes $0x'$, which has $G(x')$ representations.

		To show $G((1x)'') = G(x'') + G((x')'')$, we use the same approach. The first one either breaks or not, resulting in two different strings, $x$ and $x'$. In both cases, we must apply the double prime transformation to break a $2^0$ in order to pass it to a string on the right side of $1x$.
		
		For the equality of the bottom row, the string $(1x)'$ is $3x$; thus, the 3 must be broken, and the rest of the string becomes $x'$. So $\mu(1x) = \mu(1)\cdot\mu(x)$ holds.

		The case of $0x$ can be shown using similar conclusions. Consider
		\begin{equation*}
			\mu(0)\cdot\mu(x) =
			\begin{bmatrix}
				1 & 0\\
				1 & 1
			\end{bmatrix} \cdot
			\begin{bmatrix}
				G(x) & G(x'')\\
				G(x') & G((x')'')
			\end{bmatrix}
			=
			\begin{bmatrix}
				G(x)  & G(x'') \\
				G(x) + G(x') & G(x'') + G((x')'')
			\end{bmatrix},
		\end{equation*}
		which must equal
		\begin{equation*}
			\mu(0x) = 
			\begin{bmatrix}
				G(0x) & G((0x)'')\\
				G((0x)') & G(((0x)')'')
			\end{bmatrix} = 
			\begin{bmatrix}
				G(x) & G(x'')\\
				G(2x) & G((2x)'')
			\end{bmatrix}
			= 
			\begin{bmatrix}
				G(x) & G(x'')\\
				G(1x) & G((1x)'')
			\end{bmatrix}.
		\end{equation*}

		We have already shown $G(1x) = G(x) + G(x')$ and $G((1x)'') = G(x'') + G((x')'')$. Therefore, the equation $\mu(0x) = \mu(0)\cdot\mu(x)$ holds, and the theorem is proved.
	\end{proof}

	This theorem also gives us a helpful tool
	%??? general principle:  never praise your own results.  Here instead of 'great' you could say, for example, 'useful'.
	to compute $G(x)$, $G(x'')$, $G(x')$, and $G((x')''$ as $\mu(x)$ is just a multiplication of $\mu(1)$s and $\mu(0)$s.
	\begin{lemma} \label{G_linearization}
		For a string $x$, the equation $G(x) = \VMAT \mu(x) \WMAT $ holds. This multiplication simply returns the top-left value of the $\mu(x)$ matrix. 
	\end{lemma}

	From Theorem \ref{matrix_linearization} and Lemma \ref{G_linearization} we deduce the following result.
	%??? use the labels of theorems and lemma instead??? 
	\begin{lemma} \label{string-division} For binary strings $x, y$, the equation 
	\begin{equation}
	    G(xy) = G(x)G(y) + G(x'')G(y')
	\end{equation}
	holds.
	\end{lemma}
	\begin{proof}
	We have
		\begin{align*}
			G(xy) &= \VMAT\mu(xy)\WMAT = \VMAT\mu(x)\mu(y)\WMAT\\ 
			&= \VMAT
			\begin{bmatrix}
				G(x)G(y) + G(x'')G(y') & G(x)G(y'') + G(x'')G((y')'')\\
				G(x')G(y)+ G((x')'')G(y') & G(x')G(y'')  + G((x')'')G((y')'')
			\end{bmatrix}\WMAT \\
			&= G(x)G(y) + G(x'')G(y').
		\end{align*}
		This can also be explained in another way. If we do not break the rightmost bit of $x$, we can assume the two strings are separate and get $G(x)G(y)$ number of hyperbinary representations. In case we break it, then $G(x'')G(y')$ ways exist.
	\end{proof}

	In what follows, we always set  $v := \VMAT$ and $w := \WMAT$.

	Here we define three comparators that help us replace substrings (or contiguous subsequences) in order to obtain a new string without decreasing the string's $G$-value.

	\begin{definition}[Comparators]
		In this paper, when we state a matrix $M_1$ is greater than or equal to the matrix $M_0$, we mean each entry of $M_1 - M_0$ is non-negative (they both must share the same dimensions).
		\begin{itemize}
			\item The infix comparator: For two strings $y$ and $t$, the relation $ t \INFIX y$ holds if $\mu(t) \geq \mu(y)$ holds.
			\item The suffix comparator: For two strings $y$ and $t$, the relation $ t \SUFFIX y$ holds if $
			\mu(t)\cdot w
			\geq \mu(y)\cdot w$ holds.
			\item The prefix comparator: For two strings $y$ and $t$, the relation $t \PREFIX y$ holds if $
			v\cdot\mu(t)
			\geq v\cdot\mu(y)
			$ holds.
		\end{itemize}
	\end{definition}
	
	\begin{lemma} \label{gc_lemma}
	    If $t \INFIX y$, and $t$ represents a smaller string, no record-setter can contain $y$ as its substring.
	\end{lemma}
	\begin{proof}
		Consider a string $a = xyz$. According to Lemma \ref{G_linearization}, we have
	    \begin{equation*}
			G(a) = v \cdot \mu(x) \cdot \mu(y) \cdot \mu(z) \cdot w.
	    \end{equation*}
		Since $ \mu(t) \geq \mu(y)$, and all entries in the matrices are positive, the replacement of $y$ with $t$ does not decrease $G(a)$, and also yields a smaller number, that is $(xtz)_2 \leq (xyz)_2$. Therefore, $(xyz)_2 \notin R$.
	\end{proof}
	
	As an example, consider the two strings $111$ and $101$. Then $101 \INFIX 111$ holds, since
	\begin{equation*}
		\mu(101) =
		\begin{bmatrix}
			2 & 3\\
			1 & 2
		\end{bmatrix} \geq
		\mu(111) = 
		\begin{bmatrix}
			1 & 3\\
			0 & 1
		\end{bmatrix} .
	\end{equation*}

	\begin{lemma} \label{endLemma}
	    If $t < y$ and $t \SUFFIX y$, then $y$ is not a suffix of a record-setter.
	\end{lemma}
	\begin{proof}
		Consider a string $a = xy$. We have shown $G(a) = v \cdot \mu(x) \cdot \mu(y) \cdot w$. By
		replacing $y$ with $t$, since $\mu(t) \cdot w \geq \mu(y) \cdot w$, the value $G(a)$ does not decrease, and we obtain a smaller string.
	\end{proof}
	
	\begin{lemma} \label{beginLemma}
	    If  $t < y$ and $t \PREFIX x$, then $x$ is not a prefix of a record-setter.
	\end{lemma}
	
	\begin{corollary} \label{lemma111}
	    For an $h \in R$, since $101 \INFIX 111$, then $h$ cannot contain $111$ as a substring.
	\end{corollary}
	We have established that a record-setter $h$ cannot contain three consecutive 1s. Now, we plan to prove $h$ cannot have two consecutive 1s, either. We  do this in the following lemmas and theorems.

	The following theorem provides examples that their $G$-values equal Fibonacci numbers.
	
	\begin{theorem} \label{fibonacci-vals} For $i \geq 0$, the equations
		\begin{align}
			G((10)^i) &= F_{2i+1},\label{Fib1st} \\
			G((10)^i0) &= F_{2i + 2},\label{Fib2nd}\\
			G(1(10)^i) &= F_{2i + 2}, \text{ and}\label{Fib3rd} \\
			G(1(10)^i0) &= F_{2i + 3}\label{Fib4th}
		\end{align}
		 hold.
	\end{theorem}
	
	\begin{proof}
		We first prove that the following equation holds:
		\begin{equation}
			\mu((10)^i) = \begin{bmatrix}
				F_{2i + 1} & F_{2i}\\
				F_{2i} & F_{2i - 1}
			\end{bmatrix} .
		\label{mat10}
		\end{equation}
		The case for $i = 1$, namely $\mu(10) = \begin{bmatrix}
			2 & 1\\
			1 & 1
		\end{bmatrix}$, holds. We now use induction:
		\begin{equation*}
			\mu((10)^{i + 1}) = \mu((10)^i) \mu(10) = 
			\begin{bmatrix}
				F_{2i + 1} & F_{2i}\\
				F_{2i} & F_{2i - 1}
			\end{bmatrix}
			\begin{bmatrix}
				2 & 1\\
				1 & 1
			\end{bmatrix} =
			\begin{bmatrix}
				F_{2i + 3} & F_{2i + 2}\\
				F_{2i + 2} & F_{2i + 1}
			\end{bmatrix}, 
		\end{equation*}
		and thus we can conclude \eqref{Fib1st}.

		For the other equations \eqref{Fib2nd}, \eqref{Fib3rd}, and \eqref{Fib4th}, we proceed similarly:
		\begin{align*}
			\mu((10)^i0) = \mu((10)^i)\mu(0) = 
			\begin{bmatrix}
				F_{2i + 1} & F_{2i}\\
				F_{2i} & F_{2i - 1}
			\end{bmatrix}
			\begin{bmatrix}
				1 & 0\\
				1 & 1
			\end{bmatrix} =
			\begin{bmatrix}
				F_{2i + 2} & F_{2i}\\
				F_{2i + 1} & F_{2i - 1}
			\end{bmatrix};\\
			\mu(1(10)^i) = \mu(1)\mu((10)^i) = 
			\begin{bmatrix}
				1 & 1\\
				0 & 1
			\end{bmatrix}
			\begin{bmatrix}
				F_{2i + 1} & F_{2i}\\
				F_{2i} & F_{2i - 1}
			\end{bmatrix}
			 =
			\begin{bmatrix}
				F_{2i + 2} & F_{2i + 1}\\
				F_{2i} & F_{2i - 1}
			\end{bmatrix};\\
			\mu(1(10)^i0) = \mu(1)\mu((10)^i)\mu(0) = 
			\begin{bmatrix}
				F_{2i + 2} & F_{2i + 1}\\
				F_{2i} & F_{2i - 1}
			\end{bmatrix}
			\begin{bmatrix}
				1 & 0\\
				1 & 1
			\end{bmatrix} =
			\begin{bmatrix}
				F_{2i + 3} & F_{2i + 1}\\
				F_{2i + 1} & F_{2i - 1}
			\end{bmatrix} .
		\end{align*}
		
		Multiplying these by $v$ and $w$ as in Lemma \ref{G_linearization} confirms the equalities \eqref{Fib1st}--\eqref{Fib4th}.
\end{proof}
	
	\begin{lemma} \label{lemma1100} If $h \in R$, then $h$ cannot contain a substring of the form $1(10)^{i}0$ for $i>0$.
	\end{lemma}
	\begin{proof}
		To prove this we use Theorem \ref{fibonacci-vals} and the infix-comparator to show $t = (10)^{i+1} \INFIX y = 1(10)^{i}0$:
		\begin{equation*}
			\mu(t) = \begin{bmatrix}
				F_{2i + 3} & F_{2i + 2}\\
				F_{2i + 2} & F_{2i + 1}
			\end{bmatrix} \geq
			\mu(y) = \begin{bmatrix}
				F_{2i + 3} & F_{2i + 1}\\
				F_{2i + 1} & F_{2i - 1}
			\end{bmatrix} .
		\end{equation*}

		We conclude $t \INFIX y$ for $i \geq 1$. Because of this, a $00$ cannot appear to the right of a $11$, since if it did, it would contain a
		substring of the form $1(10)^i0$.
	\end{proof}
	
	\begin{lemma} \label{lemma110} If $h \in R$, then $h$ does not end in $1(10)^{i}$ for $i \geq 0$.
	\end{lemma}
	\begin{proof}
		Consider $t = (10)^i0$ and $y = 1(10)^{i}$. Then
		\begin{equation*}
			\mu(t) = \begin{bmatrix}
				F_{2i + 2} & F_{2i}\\
				F_{2i + 1} & F_{2i - 1}
			\end{bmatrix} \quad
			\mu(y) = \begin{bmatrix}
				F_{2i + 2} & F_{2i + 1}\\
				F_{2i} & F_{2i - 1}
			\end{bmatrix} .
		\end{equation*}
		and
		\begin{equation*}
			\mu(t)\WMAT = \begin{bmatrix}
				F_{2i + 2}\\
				F_{2i + 1}
			\end{bmatrix} \geq
			\mu(y)\WMAT = \begin{bmatrix}
				F_{2i + 2}\\
				F_{2i}
			\end{bmatrix}
		\end{equation*}
		Hence $t \SUFFIX y$, and $h$ cannot end in $y$.
	\end{proof}

	\begin{theorem}
		A record-setter $h \in R$ cannot contain the substring $11$.
	\end{theorem}
	\begin{proof}
		Suppose it does. Consider the rightmost $11$. Due to Lemma \ref{lemma1100}, there cannot be two consecutive 0s
		to its right. Therefore, the string must end in $1(10)^i$, which is impossible due to Lemma \ref{lemma110}.
	\end{proof}

	Therefore, we have shown that a record-setter $h$ is a concatenation of multiple strings of the form $1(0^i)$, for $i>0$. The next step establishes an upper bound on $i$ and shows that $i \leq 3$.

	\begin{theorem} \label{only10100} A record-setter $h \in R$ cannot contain the substring $10000$.
	\end{theorem}
	\begin{proof}
		First, we show $h$ cannot begin with $10000$:
		\begin{equation*}
			\VMAT \mu(10\ 10) = \begin{bmatrix}
				5  &  3
			\end{bmatrix}
			\geq
			\VMAT \mu(10000) = \begin{bmatrix}
				5  &  1
				\end{bmatrix}
				\Longrightarrow 10\ 10 \PREFIX 10000 .
		\end{equation*}
		
		Now consider the leftmost $10000$; it has to have a $10$, $100$, or $1000$ on its left:
		\begin{align*}
			\mu(1000\ 100) &= \begin{bmatrix}
				14  &  5 \\
				11  &  4
			\end{bmatrix}
			\geq
			\mu(10\ 10000) = \begin{bmatrix}
				14  &  3 \\
				9  &  2
				\end{bmatrix}
			&&\Longrightarrow 1000\ 100 \INFIX 10\ 10000;
			\\
			\mu(1000\ 1000) &= \begin{bmatrix}
				19  &  5 \\
				15  &  4
				\end{bmatrix}
			\geq
			\mu(100\ 10000) =
			\begin{bmatrix}
				19  &  4 \\
				14  &  3
			\end{bmatrix}
			&&\Longrightarrow 1000\ 1000 \INFIX 100\ 10000;
			\\
			\mu(100\ 100\ 10) &= \begin{bmatrix}
				26  &  15 \\
				19  &  11
				\end{bmatrix}
			\geq
			\mu(1000\ 10000) = \begin{bmatrix}
				24  &  5 \\
				19  &  4
				\end{bmatrix}
			&&\Longrightarrow 100\ 100\ 10 \INFIX 1000\ 10000 .
		\end{align*}
		Consequently, the substring $10000$ cannot appear in $h$.
	\end{proof}

    \section{Limits on the number of 1000s and 100s}\label{limit1001000} At this point, we have established that a  record-setter's binary representation consists of a concatenation of 10s, 100s, and 1000s.
	
	The following theorem limits the appearance of 1000 to the beginning of a record-setter:
	\begin{theorem} \label{begin1000} A record-setter can only have 1000 at its beginning, except in the case $1001000$.
	\end{theorem}
	\begin{proof}
		It is simple to check this condition manually for strings of length $<12$. Now, consider a record-setter $h \in R$, with $|h| \geq 12$. String $h$ must at least have three 1s.
		To prove $h$ can only have 1000 at its beginning, we use our comparators to show neither 
		\begin{itemize}
		    \item[(a)] \textcolor{blue}{101000}, nor
		    \item[(b)] \textcolor{blue}{1001000}, nor
		    \item[(c)] \textcolor{blue}{10001000}
		\end{itemize}
    	can appear in $h$.
		
		\begin{itemize}
		    \item[(a)] Consider the following comparison:
		    \begin{equation} \label{tenThousand}
    			\mu(100\ 100) = \begin{bmatrix}
    				11  &  4 \\
    				8  &  3
    				\end{bmatrix}
    			\geq
    			\mu(\textcolor{blue}{10\ 1000}) = \begin{bmatrix}
    				11  &  3 \\
    				7  &  2
    				\end{bmatrix}
    			\Longrightarrow 100\ 100 \INFIX\textcolor{blue}{10\ 1000}.
		    \end{equation}
		    We can infer that 101000 cannot appear in $h$. 
		    
		    \item[(b)] In this case, for every  $x < \textcolor{blue}{1001000}$, the equation $\mu(x) < \mu(\textcolor{blue}{1001000})$ holds, and we cannot find a replacement right away. Therefore, we divide this into two cases:
		    \begin{itemize}
		        \item[(b1)] In this case, we consider \textcolor{blue}{1001000} in the middle or at the end, thus it must have a 10, 100, or 1000 immediately on its left:
		        \begin{align}
        		\label{hundredThousand}
        			\begin{alignedat}{3}
        			    \mu( 100\ 100\ 100 ) = 
                        \begin{bmatrix}
                        41  &  15 \\
                        30  &  11
                        \end{bmatrix}
                        &\geq \
                        &\mu( 10\ \textcolor{blue}{1001000} )
                        & = \begin{bmatrix}
                            41  &  11 \\
                            26  &  7
                        \end{bmatrix},\\
        				\mu( 1000\ 10\ 10\ 10 ) = 
                        \begin{bmatrix}
                        60  &  37 \\
                        47  &  29
                        \end{bmatrix}
                        &\geq \
                        &\mu( 100\ \textcolor{blue}{1001000} )
                        & = \begin{bmatrix}
                        56  &  15 \\
                        41  &  11
                        \end{bmatrix},\\
                        \mu( 10000\ 10\ 10\ 10 ) = 
                        \begin{bmatrix}
                        73  &  45 \\
                        60  &  37
                        \end{bmatrix}
                        &\geq \
                        &\mu( 1000\ \textcolor{blue}{1001000} )
                        & = \begin{bmatrix}
                        71  &  19 \\
                        56  &  15
                        \end{bmatrix}.
        			\end{alignedat}
        		\end{align}
		        \item[(b2)] The other case would be for \textcolor{blue}{1001000} to appear at the beginning:
		        \begin{align}
        		    \label{thousandLeftHundred}
        			\begin{alignedat}{3}
        				\mu( 1000\ 110\ 10 )
                        = \begin{bmatrix}
                        35  &  22 \\
                        27  &  17
                        \end{bmatrix}
        				&\geq
        				&\ \mu( \textcolor{blue}{1001000}\ 10 )
                        = \begin{bmatrix}
                        34  &  19 \\
                        25  &  14
                        \end{bmatrix},\\
                        \mu( 1000\ 10\ 10\ 10 )
                        = \begin{bmatrix}
                        60  &  37 \\
                        47  &  29
                        \end{bmatrix}
                        &\geq
                        &\ \mu( \textcolor{blue}{1001000}\ 100 )
                        = \begin{bmatrix}
                        53  &  19 \\
                        39  &  14
                        \end{bmatrix},\\
        				\mu( 100\ 10\ 10\ 100 )
                        = \begin{bmatrix}
                        76  &  29 \\
                        55  &  21
                        \end{bmatrix}
                        &\geq
                        &\ \mu( \textcolor{blue}{1001000}\ 1000 )
                        = \begin{bmatrix}
                        72  &  19 \\
                        53  &  14
                        \end{bmatrix}.
        			\end{alignedat}
        		\end{align}
		    \end{itemize}
		    Therefore $h$ cannot contain \textcolor{blue}{1001000}.
		    
		    \item[(c)] Just like the previous case, there is no immediate replacement for \textcolor{blue}{10001000}. We divide this into two cases:
		    \begin{itemize}
		        \item[(c1)] There is a prefix replacement for \textcolor{blue}{10001000}:
		        \begin{multline}
		            v. \mu( 10\ 100\ 10 )
                    = \begin{bmatrix}
                    19  &  11
                    \end{bmatrix}
                    \geq
                    v.\mu( \textcolor{blue}{10001000} )
                    = \begin{bmatrix}
                    19  &  5
                    \end{bmatrix}\\ \Longrightarrow 10\ 100\ 10 \PREFIX \textcolor{blue}{10001000}.
		        \end{multline}
		        
		        \item[(c2)] In case \textcolor{blue}{10001000} does not appear at the beginning, there must be a 10, 100, or a 1000 immediately on its left:
		        \begin{align}
        		\label{thousandThousand}
        			\begin{alignedat}{3}
        			\mu( 10\ 10\ 10\ 100 ) = 
                    \begin{bmatrix}
                    55  &  21 \\
                    34  &  13
                    \end{bmatrix}
                    &\geq\
                    &\mu( 10\ \textcolor{blue}{10001000} )
                    & = \begin{bmatrix}
                    53  &  14 \\
                    34  &  9
                    \end{bmatrix},\\
                    \mu( 100\ 10\ 10\ 100 )
                    = \begin{bmatrix}
                    76  &  29 \\
                    55  &  21
                    \end{bmatrix}
                    &\geq\
                    &\mu( 100\ \textcolor{blue}{10001000} )
                    &= \begin{bmatrix}
                    72  &  19 \\
                    53  &  14
                    \end{bmatrix},\\
                    \text{and }\mu( 1000\ 10\ 10\ 100 )
                    = \begin{bmatrix}
                    97  &  37 \\
                    76  &  29
                    \end{bmatrix}
        			&\geq\
        			&\mu( 1000\ \textcolor{blue}{10001000} )
                    &= \begin{bmatrix}
                    91  &  24 \\
                    72  &  19
                    \end{bmatrix}.
        			\end{alignedat}
        		\end{align}
		    \end{itemize}
		\end{itemize}
% 		??? when you talk about strings x to the left of y or to the right of y, please clarify somewhere whether you mean immediately to the left or right or not ???
		
% 		??? consider whether or not to use `regular expression' notation for some of your assertions
% 		e.g.
% 		$10^i$ for $i \geq 0$ you could say $10^*$. It might make things more precise in some places.
	\end{proof}
	
	Considering Theorem \ref{begin1000}, we can easily guess that 1000s do not often appear in record-setters. In fact, they only appear once for each length. We will prove this result later in Lemmas \ref{even1000} and \ref{odd1000}, but for now, let us consider that our strings only consist of 10s and 100s.
	
	The plan from here onward is to limit the number of 100s. The next set of theorems and lemmas concerns this limitation. To do this, we calculate the maximum $G$-values for strings with $0, 1, \ldots, 5$  100s and compare them. Let $h$ be a string; we define the function $\delta(h)$ as the difference between the number of 0s and 1s occurring in $h$. For strings only containing 100s and 10s, the quantity $\delta(h)$ equals the number of 100s in $h$.

	The following theorem was previously proved in \cite{Lucas:1878}:
	\begin{theorem} \label{max-val-prime}
		The maximum $G$-value for strings of length $2n$ $(s(t)$ for $ 2^{2n-1} \leq t < 2^{2n})$ is $F_{2n + 1}$, and it first
		appears in the record-setter $(10)^n$.
		
		The maximum $G$-value for strings of length $2n + 1$ $(s(t)$ for $ 2^{2n} \leq t < 2^{2n + 1})$ is $F_{2n + 2}$, and it first
		appears in the record-setter $(10)^n0$.
	\end{theorem}

	The above theorem represents two sets of strings $(10)^+$ and $(10)^+0$, with $\delta$-values 0 and 1.
	
% 	??? what does x refer to here ???

	\begin{lemma} \label{replace10}
	    Consider a string $yz$, where $z$ begins with 1. If $|z| = 2n$ for $n \geq 1$, then $G(y (10)^{2n}) \geq G(yz)$. If $|z| = 2n + 1$, then $G(y (10)^{2n}0) \geq G(yz)$.
	\end{lemma}
	
% 	??? Maybe replace the first statement with
% 	something like the following:  ``Write $x = yz$ with $|z| = 2n$, where $z$ begins with $1$.
% 	Then $G(y (10)^{2n}) \geq G(x)$.'' ???
	
	\begin{proof}
		Consider the matrix $\mu((10)^n)\WMAT = \begin{bmatrix}
			F_{2n + 1}\\
			F_{2n}
		\end{bmatrix}$. The suffix matrix for $z$ is $\mu(z)\WMAT = \begin{bmatrix}
			G(z)\\
			G(z')
		\end{bmatrix}$. Since $F_{2n + 1} \geq G(z)$, and $|z'| < |z|$ (since $z$ begins with 1), the value of $G(z')$ cannot exceed $F_{2n}$. Therefore $(10)^n \SUFFIX z$.

		For an odd length $2n + 1$, with the same approach, the matrix $\mu((10)^n0)\WMAT = \begin{bmatrix}
			F_{2n + 2}\\
			F_{2n + 1}
		\end{bmatrix} \geq \mu(z)\WMAT = \begin{bmatrix}
			G(z)\\
			G(z')
		\end{bmatrix}$, and $z$ can be replaced with $(10)^n0$.
	\end{proof}

	To continue our proofs, we need simple lemmas regarding the Fibonacci sequence:
	\begin{lemma} \label{oddFibZero}
		The sequence $F_1F_{2n}$, $F_3F_{2n - 2}$, \ldots, $F_{2n-1}F_2$ is strictly decreasing.
	\end{lemma}
	\begin{proof}
	    Consider an element of the sequence $F_{2i+1}F_{2n - 2i}$. There are two cases to consider, depending on the relative magnitude of $n$ and $2i$.
	    
		If $n \geq 2i + 1$, then
		\begin{align*}
			F_{2i + 1}F_{2n - 2i} &= F_{2i + 2}F_{2n - 2i} - F_{2i}F_{2n - 2i} = F^2_{n + 1} - F^2_{n - 2i - 1} - F^2_n + F^2_{n - 2i}\\ &= (F^2_{n+1} - F^2_{n}) + (F^2_{n - 2i} - F^2_{n - 2i - 1}).
		\end{align*}
		Notice that the first term,
		namely $(F_{n+1}^2 -F_n^2)$ is a constant, while the second term
		 $F^2_{n - 2i} - F^2_{n - 2i - 1} = F_{n - 2i - 2}F_{n - 2i + 1}$ decreases
		 with an increasing $i$.

		If $n \leq 2i$, then
		\begin{equation*}
			F_{2i + 1}F_{2n - 2i} = (F^2_{n+1} - F^2_{n}) + (F^2_{2i - n} - F^2_{2i + 1 - n}).
		\end{equation*}
		The non-constant term is $F^2_{2i - n} - F^2_{2i + 1 - n} = -F_{2i - n - 1}F_{2i + 2 - n}$, which is negative and still decreases.		
	\end{proof}
	
% 	??? fix "differentiating factor"
	
	\begin{lemma} \label{evenMult}
	    The sequence $F_0F_{2n}$, $F_2F_{2n - 2}$, \ldots, $F_nF_n$ is strictly increasing.
	\end{lemma}
	\begin{proof}
		For $0 \leq i \leq n/2$, We already know that $F_{2i}F_{2n - 2i} = F^2_n - F^2_{n - 2i}$. Since the sequence $F^2_n$, $F^2_{2n - 2}$, \ldots, $F^2_0$ decreases, the lemma holds.
	\end{proof}
	
% 	We now define a useful tool to ease our calculations
% 	\begin{definition} [Passers and Dead ends]
% 	    We define \textcolor{green}{green 0s} as ``passers''. They never take the value of 2. If they ever receive 2, they immediately break it down, as if they never existed (\ref{green0s}). 
	    
% 	    \textcolor{red}{Red 2s} are named as ``dead ends''. They never break, as if the string is divided into two separate pieces (\ref{red2s}).
% 	\end{definition}
	
% 	\begin{lemma} \label{green0s}
% 	    The equation $G(xy) = G(x\textcolor{green}{0}y)$ holds.
% 	\end{lemma}
% 	\begin{proof}
% 	    According to Lemma \ref{string-division}, we already know $G(xy) = G(x)G(y) + G(x'$
% 	\end{proof}
	
% 	\begin{lemma} \label{red2s}
% 	    The equation $G(x0y) = G(x\textcolor{green}{0}y) + G(x\textcolor{red}{2}y)$ holds.
% 	\end{lemma}

	In the next theorem, we  calculate the maximum $G$-value obtained by a string $x$ with $\delta(x) = 2$.
	\begin{lemma} [Strings with two 100s] \label{two100s}
		The maximum $G$-value for strings with two 100s occurs for $(10)^n0(10)^{n-1}0$ for lengths $l = 4n$, or
		for $(10)^{n}0(10)^{n}0$ for lengths $l = 4n + 2$, while $l \geq 6$.
	\end{lemma}
	\begin{proof}
		
% 		??? Quantify i, j, k as $\geq 1$, $\geq 0$, or what?
% 		??? You are describing a process here,
% 		say it more clearly ??? envision a process where the bit does not take the value of two.
        To simplify the statements, we write $\mu(10) = \mu(1)\mu(0)$ as $\mu_{10}$, and $\mu(0)$ as $I_2 + \gamma_0$, where $$I_2 = \IMAT, \text{ and } \gamma_0 = \ZMAT.$$
        
        Consider the string $(10)^i0 (10)^j0(10)^k$, where $i,j \geq 1$ and $k \geq 0$:
		\begin{align*}
		    G((10)^i0(10)^j0(10)^k) = v\mu^i_{10}\mu(0)\mu^j_{10}\mu(0)\mu^k_{10}w
		    = v\mu^i_{10}(I + \gamma_0)\mu^j_{10}(I + \gamma_0)\mu^{k}_{10}w\\
		    = v\mu^{i + j + k}_{10}w +
		    v\mu^i_{10}\gamma_0\mu^{j + k}_{10}w +
		    v\mu^{i + j}_{10}\gamma_0\mu^k_{10}w +
		    v\mu^i_{10}\gamma_0\mu^j_{10}\gamma_0\mu^k_{10}w.
		\end{align*}
		We now evaluate each summand in terms of Fibonacci numbers.
		\begin{align*}
			v\mu^{i + j + k}_{10}w &= v\begin{bmatrix}
			    F_{2i + 2j + 2k + 1} & F_{2i + 2j + 2k}\\
			    F_{2i + 2k + 2k} & F_{2i + 2j + 2k - 1}
			\end{bmatrix}w = F_{2i + 2j + 2k + 1} \\
			v\mu^i_{10}\gamma_0\mu^{j + k}_{10}w  &= \begin{bmatrix}
			    F_{2i + 1} & F_{2i}
			\end{bmatrix}
			\ZMAT
			\begin{bmatrix}
			    F_{2j + 2k + 1}\\
			    F_{2j + 2k}
			\end{bmatrix} = F_{2i}F_{2j + 2k + 1}
			\\
			v\mu^{i + j}_{10}\gamma_0\mu^k_{10}w &= \begin{bmatrix}
			    F_{2i + 2j + 1} & F_{2i + 2j}
			\end{bmatrix}
			\ZMAT
			\begin{bmatrix}
			    F_{2k + 1}\\
			    F_{2k}
			\end{bmatrix} = F_{2i+2j}F_{2k + 1}\\
			v\mu^i_{10}\gamma_0\mu^j_{10}\gamma_0\mu^k_{10}w &=
			\begin{bmatrix}
			    F_{2i + 1} & F_{2i}
			\end{bmatrix} 
			\ZMAT
			\begin{bmatrix}
			    F_{2j + 1} & F_{2j}\\
			    F_{2j} & F_{2j - 1}
			\end{bmatrix} 
			\ZMAT
			\begin{bmatrix}
			    F_{2k + 1}\\
			    F_{2k}
			\end{bmatrix} = F_{2i}F_{2j}F_{2k + 1} .
		\end{align*}

		For a fixed $i$, according to Lemma \ref{oddFibZero}, to maximize the above equations $k := 0$ must become zero, and $j := j + k$. Then the above equation can be written as
		\begin{equation*}
			G((10)^i0(10)^j0) = v\mu^i_{10}I_2\mu^j_{10}\mu(0)w + v\mu^i_{10}\gamma_0\mu^j_{10}\mu(0)w = F_{2i + 2j + 2} + F_{2i}F_{2j + 2}.
		\end{equation*}

		In case $l = 4n = 2i + 2j + 2$, to maximize the above equation, according to Lemma \ref{evenMult}, $i = n$, $j = n-1$, and the $G$-value would be $F_{4n} + F^2_{2n}$. In case $l = 4n + 2$, $i = j = n$, and the $G$-value is $F_{4n + 2} + F_{2n}F_{2n + 2} = F_{4n + 2} + F^2_{2n + 1} - 1$. Thus the theorem holds. Also in general, for any even $l$, the maximum $G$-value $\leq F_{l} + F^2_{l/2}$. 
	\end{proof}
	
	\begin{lemma} \label{minValSingle100}
		Let $x = (10)^i0(10)^{n - i}$ be a string of length $2n + 1$ for $n \geq 1$ and $i\geq 1$ containing a single 100. Then, the minimum $G$-value for $x$ is $F_{2n + 1} + F_{2n - 1}$.
	\end{lemma}
	\begin{proof}
	We have
		\begin{align*}
			G(x) = G((10)^i0(10)^{n - i}) = v \cdot \mu^i_{10} \cdot (I + \gamma_0) \cdot \mu^{n - i}_{10} \cdot w = F_{2n + 1} + F_{2i}F_{2n-2i+1} \\
			\xRightarrow{{\rm Thm.}~\ref{oddFibZero}\ i = 1\ } F_{2n + 1} + F_{2n - 1}.
		\end{align*}
	\end{proof}
	
	\begin{theorem} \label{three100s}
		For two strings $x$ and $y$, if $\delta(x) = 3$ and $\delta(y) = 1$ then $G(x) < G(y)$.
	\end{theorem}
	% 1000's appearance is already limited. This theorem will later be used to limit the appearance of 100s.
	\begin{proof}
		Consider the two strings of the same length below:
		\begin{center}
			\begin{tabular}{ll}
				 $x = (10)^i100$ \fbox{$(10)^j0(10)^{k-1-j}0$}
				\\
				$y= 100(10)^i$  \fbox{$(10)^{k}$} .
			\end{tabular}
		\end{center}
		We must prove for $i \geq 0$, $j \geq 1$, and $k - 1 - j \geq 1$, the inequality $G(x) \leq G(y)$ holds, where $y$ has the minimum $G$-value among the strings with a single 100 (see Lemma \ref{minValSingle100}). 
		\begin{align*}
			G(x) &= G((10)^i100)G((10)^j0(10)^{k-1-j}0) +
		G((10)^i0)G(1(10)^{j-1}0(10)^{k-1-j}0)\\
			&\leq F_{2i + 4} (F^2_k + F_{2k}) + F_{2i + 2}F_{2k} =
			F_{2i + 4} \left(\dfrac{2F_{2k + 1} - F_{2k} - 2}{5} + F_{2k} \right) + F_{2i + 2}F_{2k} .\\
			G(y) &= G(100(10)^i)F_{2k + 1} +
			G(100(10)^{i-1}0)F_{2k} \\
			&= (F_{2i+3} + F_{2i + 1})F_{2k + 1} + (F_{2i} + F_{2i + 2})F_{2k}  \\
			&= (F_{2i+4} - F_{2i})F_{2k + 1} + (F_{2i} + F_{2i + 2})F_{2k}.
		\end{align*}
		
		We now show $G(y) - G(x) \geq 0$:
		\begin{multline}
			G(y) - G(x) \geq (F_{2i+4} - F_{2i})F_{2k + 1} + (F_{2i} +  \cancel{F_{2i + 2}})F_{2k} - F_{2i + 4} \left(\dfrac{2F_{2k + 1} + 4F_{2k} - 2}{5} \right) - \cancel{F_{2i + 2}F_{2k}} \\
			\begin{aligned}
				\xRightarrow{\times 5}
				5F_{2i + 4}F_{2k + 1} - 5F_{2i}F_{2k - 1}
				- 2F_{2i + 4}F_{2k + 1} - 4F_{2i + 4}F_{2k} + 2F_{2i + 4} &\\= 
				F_{2i + 4}(3F_{2k + 1} - 4F_{2k} + 2) - 5F_{2i}F_{2k - 1} &\\=
				F_{2i + 4}(F_{2k - 1} + F_{2k - 3} + 2) - 5F_{2i}F_{2k - 1} &\\=
				F_{2i + 4}(F_{2k - 3} + 2) + F_{2k - 1}(F_{2i+4} - 5F_{2i}) &\\=
				F_{2i + 4}(F_{2k - 3} + 2) + F_{2k - 1}(\cancel{5F_{2i}} + 3F_{2i-1} - \cancel{5F_{2i}}) &\geq 0.
			\end{aligned}
		\end{multline}
	\end{proof}

	Theorem \ref{three100s} can be generalized for all odd number occurrences of 100. To do this, replace the right side of the third 100 occurring in $x$ using Lemma~\ref{replace10}.
	
	\begin{lemma} \label{replaceWith10010}
		Let $i \geq 1$, and let $x$ be a string with $|x| = 2i + 3$ and $\delta(x) = 3$.
		Then $y = 100(10)^i \SUFFIX x$.
	\end{lemma}
	\begin{proof}
		We have already shown that $G(y) > G(x)$ (Theorem~\ref{three100s}). Also, the inequality $G(y') > G(x')$ holds since $y' = (10)^{i + 1}$, and $G(y')$ is the maximum possible $G$-value for strings of length $2i + 2$.
	\end{proof}

	\begin{theorem} \label{noFour100s}
		Let $n \geq 4$.  If $|x| = 2n + 4$ and $\delta(x) = 4$, then $x \notin R$.
	\end{theorem}
	
% 	??? I rephrased statement of theorem, check if the domain of n is correct .
	
	\begin{proof}
		Consider three cases where $x$ begins with a 10, a 10010, or a 100100.
		If $x$ begins with $10$, due to Lemma \ref{replaceWith10010}, we can replace the right side of the first 100, with $100(10)^*$, and get the string $y$. For example, if $x = $ 10 10 \textcolor{blue}{100} 10 100 100 100 becomes $y = $ 10 10 \textcolor{blue}{100} \textcolor{blue}{100} 10 10 10 10, which has a greater $G$-value.
		
		Then consider the strings $a = 10\ 100\ 100\ (10)^i$ and $b = 100\ 10\ (10)^i\ 100$:
		\begin{align*}
			\mu(a)\WMAT &= \begin{bmatrix}
				30  &  11 \\
				19  &  7
				\end{bmatrix}
				\begin{bmatrix}
					F_{2i + 1}\\
					F_{2i}
				\end{bmatrix} = 
				\begin{bmatrix}
					30F_{2i + 1} + 11F_{2i}\\
					19F_{2i + 1} + 7F_{2i}
				\end{bmatrix} = 
				\begin{bmatrix}
					11F_{2i + 3} + 8F_{2i + 1}\\
					7F_{2i + 3} + 5F_{2i + 1}
				\end{bmatrix}\\
		\mu(b)\WMAT &= \begin{bmatrix}
				7  &  4 \\
				5  &  3
				\end{bmatrix}
				\begin{bmatrix}
					F_{2i + 4}\\
					F_{2i + 3}
				\end{bmatrix} = 
				\begin{bmatrix}
					7F_{2i + 4} + 4F_{2i + 3}\\
					5F_{2i + 4} + 3F_{2i + 3}
				\end{bmatrix},
		\end{align*}
		so $b \SUFFIX a$ for $i \geq 1$. Therefore, by replacing suffix $a$ with $b$, we get a smaller string with a greater $G$-value. So $x \notin R$.

		Now consider the case where $x$ begins with 10010. Replace the 1st 100's right with $100(10)^{n - 1}$, so that we get $100\ 100\ (10)^{n-1}$. After these replacements, the $G$-value does not decrease, and we also get smaller strings.
		
		The only remaining case has $x$ with two 100s at the beginning. We compare $x$ with a string beginning with 1000, which is smaller. Let $x_2$ represent the string $x$'s suffix of length $2n - 2$, with two 100s. The upper bound on $G(x_2)$ and $G(10x_2)$ is achieved using Lemma \ref{two100s}:
		\begin{equation*}
			G(x) = G(1001\ 00 x_2) = G(1001)G(00x_2) + G(1000)G(10x_2)
			\leq 3(F_{2n-2} + F^2_{n - 1}) + 4(F_{2n} + F^2_n) .
		\end{equation*}
		
		After rewriting the equation to swap $F^2$s with first order $F$, multiply the equation by 5 to remove the $\dfrac{1}{5}$ factor:
		\begin{equation*}
			3(2F_{2n -1} + 4F_{2n - 2} - 2)+ 4(2F_{2n + 1} + 4F_{2n} + 2) = 8F_{2n + 2} + 14F_{2n} + 6F_{2n - 2} + 2
		\end{equation*}

		We now compare this value with $5G(1000\ (10)^n)$:
		\begin{align*}
			5G(1000\ (10)^n) = 20F_{2n + 1} + 5F_{2n}\\
			20F_{2n + 1} + 5F_{2n} &\geq 8F_{2n + 2} + 14F_{2n} + 6F_{2n - 2} + 2 \\ \rightarrow
			12F_{2n + 1} &\geq 17F_{2n} + 6F_{2n - 2} + 2\\ \rightarrow
			12F_{2n - 1} &\geq 5F_{2n} + 6F_{2n - 2} + 2
			\\ \rightarrow
			7F_{2n - 1} &\geq 11F_{2n - 2} + 2 \\ \rightarrow
			7F_{2n - 3} &\geq 4F_{2n - 2} + 2 \\ \rightarrow
			3F_{2n - 3} &\geq 4F_{2n - 4} + 2 \\ \rightarrow
			2F_{2n - 5} &\geq F_{2n - 6} + 2,
		\end{align*}
		which holds for $n \geq 4$. Therefore we cannot have four 100s in a record-setter. For six or more 100s, the same proof can be applied by replacing the fourth 100's right with 10s using Theorem~\ref{replace10}.
	\end{proof}
	
	\begin{theorem} \label{even1000}
		For even lengths $2n + 4$ with $n \geq 0$, only a single record-setter $h$ beginning with 1000 exists. String $h$ is also the first record-setter of length $2n + 4$.
	\end{theorem}
	\begin{proof}
		The only record-setter is $h = 1000\ (10)^n$. Let $x$ be a string with length $|x| = 2n$ containing 100 substrings ($n$ must be $\geq 3$ to be able to contain 100s). Using Lemma \ref{two100s}:
		\begin{equation*}
			5G(1000\ x) \leq 4(5F^2_n + 5F_{2n}) + 5F_{2n}
			\leq 8F_{2n + 1} + 21F_{2n} + 8 \leq 5F_{2n + 4}.
		\end{equation*}
		The above equation holds for $n \geq 5$. For $n = 4$:
		\begin{equation*}
			5G(1000\ x) \leq 4(F^2_4 + F_{8}) + F_{8} = 141 \leq F_{12} = 144.
		\end{equation*}
		For $n = 3$:
		\begin{equation*}
			G(1000\ 100\ 100) = 52 \leq G(101010100) = 55.
		\end{equation*}
		Ergo, the $G$-value cannot exceed $F_{2n + 4}$, which the smaller string $(10)^{n + 1}0$ already has.

		Let us calculate $G(h)$:
		\begin{align*}
			G(1000\ (10)^{n}) = 4F_{2n + 1} + F_{2n} = F_{2n + 2} + 3F_{2n + 1}\\ = F_{2n + 3} + 2F_{2n + 1} > F_{2n + 3} + F_{2n + 2} = F_{2n + 4} .
		\end{align*}
		Hence, the string $1000\ (10)^{n}$ is the first record-setter of length $2n + 4$ with a $G$-value greater than $F_{2n + 4}$, which is the maximum (Theorem~\ref{max-val-prime}) generated by the strings of length $2n + 3$. This makes $h$ the first record-setter of length $2n + 4$.
	\end{proof}

	\begin{theorem}
		Let $x$ be a string with length $|x| = 2n + 9$, for $n \geq 3$, and $\delta(x) \geq 5$. Then $x \notin R$.
	\end{theorem}
% 	??? quantify n >= what?
	
	\begin{proof}
		Our proof provides smaller strings with greater $G$-values only based on the position of the first five 100s. So for cases where $\delta(x) \geq 7$, replace the right side of the fifth 100 with 10s (Lemma~\ref{replace10}). Therefore consider $\delta(x)$ as 5, and $x = (10)^i0\ (10)^j0\ (10)^k0\ (10)^p0\ (10)^q0\ (10)^r$,
		with $i,j,k,p,q, \geq 1$ and $r \geq 0$.
		
		First, we prove that if $i = 1, j = 1, k = 1$ does not hold, then $x \notin R$.
    \begin{itemize}
%??? compare it what?
		\item[(a)] If $i>1$, then smaller string $100(10)^{n + 3}$ has a greater $G$-value as proved in Lemma \ref{three100s}.
        
        \item[(b)] If $j > 1$, using the approach as in Theorem~\ref{noFour100s}, we can obtain a smaller string with a greater $G$-value.
        
        \item[(c)] If $k > 1$, using Lemma~\ref{replaceWith10010}, by replacing $(10)^k0\ (10)^p0\ (10)^q0\ (10)^r$ with $100\ (10)^{n + 1 - j}$, we obtain $y$ with $G(y) > G(x)$.
        
	\end{itemize}
% 	We compare it (??? what?) with a string beginning with a 
		Now consider the case where $i = 1$, $j = 1$, $k = 1$. Let $x_2$, with $|x_2| = 2n$, be a string with two 100s:
		\begin{align*}
			&G(100100100\ x_2) = 41(F^2_n + F_{2n}) + 15F_{2n} \leq 16.4F_{2n + 1} + 47.8F_{2n} + 16.4\\
			&G(100010101\ 0(10)^{n-1}0) = 23F_{2n} + 37F_{2n + 1}\\
			&23F_{2n} + 37F_{2n + 1} - 16.4F_{2n + 1} - 47.8F_{2n} - 16.4 \geq 20F_{2n + 1} -25F_{2n} - 17 \geq 0
		\end{align*}
		The above equation holds for $n \geq 2$.
		
	\end{proof}
	
	\begin{theorem} \label{odd1000}
		For odd lengths $2n + 5$ with $n \geq 1$, only a single record-setter $h$ beginning with 1000 exists. String $h$ is also the first record-setter of length $2n+5$.
	\end{theorem}
	\begin{proof}
		The first record-setter is $h = 1000\ (10)^n0$. Consider another string $1000x$. If $x$ has three or more occurrences of 100, then Lemma \ref{three100s} showed that $1000\ 100\ (10)^{n - 1}$ has a greater $G$-value. Therefore it is enough to consider strings $x$s with a single 100. Suppose $1000x = 1000\ (10)^{n-i}0(10)^i$, with $i \geq 1$:
		\begin{equation*}
			G(1000\ (10)^{n-i}0(10)^i) = 4G((10)^{n-i}0(10)^i) +
			G(1(10)^{n - i - 1}0(10)^i) .
		\end{equation*}
		
		We now evaluate $G((10)^{n-i}0(10)^i)$:
		\begin{align*}
			v\mu^{n}_{10}w &= v\begin{bmatrix}
			    F_{2n + 1} & F_{2n}\\
			    F_{2n} & F_{2n - 1}
			\end{bmatrix}w = F_{2n + 1} \\
			v\mu^{n - i}_{10}\gamma_0\mu^{i}_{10}w  &= \begin{bmatrix}
			    F_{2n - 2i + 1} & F_{2n - 2i}
			\end{bmatrix}
			\ZMAT
			\begin{bmatrix}
			    F_{2i + 1}\\
			    F_{2i}
			\end{bmatrix} = F_{2n - 2i}F_{2i + 1}\\
			\Longrightarrow\ G((10)^{n-i}0(10)^i) &= 
			v\mu^{n - i}_{10}(I_2 + \gamma_0)\mu^{i}_{10}w = F_{2n + 1} + F_{2n - 2i}F_{2i + 1}.
		\end{align*}
		
	     Next, we evaluate $G(1(10)^{n - i - 1}0(10)^i)$:
	    \begin{align*}
			v\mu(1)\mu^{n - 1}_{10}w &= v\begin{bmatrix}
			    F_{2n} & F_{2n - 1}\\
			    F_{2n - 2} & F_{2n - 3}
			\end{bmatrix}w = F_{2n} \\
			v\mu(1)\mu^{n - i - 1}_{10}\gamma_0\mu^{i}_{10}w  &= \begin{bmatrix}
			    F_{2n - 2i} & F_{2n - 2i - 1}
			\end{bmatrix}
			\ZMAT
			\begin{bmatrix}
			    F_{2i + 1}\\
			    F_{2i}
			\end{bmatrix} = F_{2n - 2i - 1}F_{2i + 1}\\
			\Longrightarrow\ G(1(10)^{n - i - 1}0(10)^i) &= 
			v\mu(1)\mu^{n - i - 1}_{10}(I_2 + \gamma_0)\mu^{i}_{10}w = F_{2n} +  F_{2n - 2i - 1}F_{2i + 1}.
		\end{align*}
		
		We can now determine $G(1000\ (10)^{n-i}0(10)^i)$:
		\begin{align*}
			G(1000\ (10)^{n-i}0(10)^i) = 4F_{2n + 1} + 4F_{2n - 2i}F_{2i + 1} + F_{2n} + F_{2n - 2i - 1}F_{2i + 1}\\ =
			4F_{2n + 1} + F_{2n} + F_{2i + 1}(4F_{2n - 2i} + F_{2n - 2i - 1})\\ =
			4F_{2n + 1} + F_{2n}+ F_{2i + 1}(2F_{2n - 2i} + F_{2n - 2i + 2}).
		\end{align*}
		To maximize this, we need to make $i$ as small as possible:
		\begin{equation*}
			4F_{2n + 1} + F_{2n}+ F_{3}(2F_{2n - 2} + F_{2n}) = 4F_{2n + 1} + 3F_{2n} + 4F_{2n - 2} < F_{2n + 5},
		\end{equation*}
		which is less than $G((10)^{n + 2}) = F_{2n + 5}$. For $h$ we have
		\begin{align*}
			G(1000\ (10)^{n}0) = 4G((10)^{n}0) +
			G(1(10)^{n - 1}0) = 4F_{2n + 2} + F_{2n + 1} \\= F_{2n + 3} + 3F_{2n + 2} = F_{2n + 4} + 2F_{2n + 2} > F_{2n + 4} + F_{2n + 3} = F_{2n + 5}.
		\end{align*}
		Therefore, the string $1000\ (10)^n0$ is the only record-setter beginning with 1000. Also, since string $h$ begins with 1000 instead of 100 or 10, it is the first record-setter of length $2n + 5$.
	\end{proof}

	\section{Record-setters of even length} \label{final_even}
	
% 	??? Record-setters of even length ???
% 	change title of section

	At this point, we have excluded all the possibilities of a record-setter $x$ with $\delta(x) > 2$, because if the string $x$ contains 1000, by Theorem~\ref{even1000}, the only record-setter  has a $\delta$-value of 2. Also, $x$ cannot have more than two 100s according to Theorem~\ref{noFour100s}. We determine the set of records-setters with $\delta(x) = 2$. If $\delta(x) = 0$, then $x$ has the maximum $G$-value as shown in Theorem \ref{max-val-prime}.

	% \begin{align*}
	% 	G( (10)^i0 (10)^j0 (10)^k ) = F_{2i + 2j + 2k + 1} + 
	% 	F_{2i + 2j}F_{2k + 1}\\ +
	% 	F_{2i}F_{2j + 2k + 1} +
	% 	F_{2i}F_{2j}F_{2k + 1}
	% \end{align*}

	\begin{theorem} \label{evenBeginEnd}
		Let $x = (10)^i0 (10)^j0 (10)^k$ be a string with two 100s, that is $i,j \geq 1$ and $k \geq 0$. If $i > 1$ and $k > 0$, then $x\notin R$.
	\end{theorem}
	\begin{proof}
		For a fixed $i$, to maximize the $G$-value, we must minimize $k$ and add its value to $j$.
		\begin{align*}
			G((10)^i0 (10)^j0 (10)) = F_{2i + 2j + 3} + 
			F_{2i + 2j}F_{3} +
			F_{2i}F_{2j + 3} +
			F_{2i}F_{2j}F_{3}\\
			= F_{2i + 2j + 3} + 2F_{2i + 2j} + F_{2i}F_{2j + 3} + 2F_{2i}F_{2j}.
		\end{align*}
		We now compare this value with a smaller string:
		$$
		G((10)^{i-1}0 (10)^{j+2}0) = F_{2i + 2j + 3} + 
		F_{2i + 2j + 2} +
		F_{2i - 2}F_{2j + 5} +
		F_{2i - 2}F_{2j + 4}$$
		and
		\begin{align}\label{equationIJK}
		\begin{split}
		    &G((10)^{i-1}0 (10)^{j+2}0) - G((10)^i0 (10)^j0 (10)) \\
			&= \cancel{F_{2i + 2j + 3}} + 
			F_{2i + 2j + 2} +
			F_{2i - 2}F_{2j + 6} - 
			\cancel{F_{2i + 2j + 3}} - 2F_{2i + 2j} - F_{2i}(F_{2j + 3}
			+ 2F_{2j})  \\
			&=F_{2i + 2j - 1} + F_{2i - 2}F_{2j + 6} - F_{2i}(F_{2j + 3}
			+ 2F_{2j}) \\
			&=F_{2i + 2j - 1} + F_{2i - 2}(F_{2j + 6} - F_{2j + 3}
			- 2F_{2j}) - F_{2i - 1}(F_{2j + 3} + 2F_{2j})  \\
			&=F_{2i + 2j - 1} + 2F_{2i - 2}(F_{2j + 3} + F_{2j + 1}) - F_{2i - 1}(F_{2j + 3} + 2F_{2j})  \\
			&=F_{2i + 2j - 1} + F_{2j + 3}F_{2i - 4} + 2F_{2i - 2}F_{2j + 1} - 2F_{2i - 1}F_{2j} \geq 0.
		\end{split}
		\end{align}
% 		the above equation
% 		??? which one ???
		Since $i \geq 2$, Equation \eqref{equationIJK} holds, resulting in finding a smaller string with a greater $G$-value.
	\end{proof}
	
	\begin{theorem}
% 	??? quantify $n$ ???
		The record-setters of even length $2n + 2$,
		for $n \geq 5$, are as follows:
		$$\begin{cases}
			1000\ (10)^{n - 1},\\
			100\ (10)^{i+1}0\ (10)^{n - i - 2}, &\text{ for } 0 \leq i \leq n - 2, \\
			(10)^i0\ (10)^{n - i}0,  & \text{ for } 1 < i \leq \lceil\frac{n}{2}\rceil ,\\
			(10)^{n + 1}.
		\end{cases}$$
		\label{eventhm}
	\end{theorem}
	
	\begin{proof}
		According to Theorem \ref{even1000}, the set of record-setters for length $2n + 2$ starts with $1000 (10)^{n - 1}$.
		
		Next, we consider the strings beginning with 100 and using the same solution as in Theorem \ref{odd1000}, the $G$-value increases as the 2nd 100 moves to the end.
		
		Theorem \ref{evenBeginEnd} proved if a record-setter does not begin with a 100, it must end in one. In Lemmas \ref{two100s} and \ref{evenMult} we showed $G((10)^i0 (10)^{n - i}0) = F_{2n + 2} + F_{2i}F_{2n - 2i + 2}$ and it would increase until $i \leq \lceil\frac{n}{2}\rceil$.
		
		In Theorem \ref{max-val-prime}, it was shown that the smallest string with the maximum $G$-value for strings of length $2n+2$ is $(10)^{n + 1}$ and $G((10)^{n + 1}) = F_{2n + 3}$.
	\end{proof}

	\section{Record-setters of odd length} \label{final_odd}
% 	??? again change section title like previous section ???
	For strings $x$ of odd length, we have established that if $\delta(x) > 3$, then $x \notin R$. We now look for the final set of record-setters, among the remaining possibilities, with either $\delta(x) = 3$ or $\delta(x) = 1$.

	% Consider the strings with three 100s, as we know they must begin with a 100:
	% \begin{align*}
	% 	G(1\ 00 (10)^i0 (10)^j0 (10)^k ) = G((10)^i0 (10)^j0 (10)^k)
	% 	+ G((10)^{i+1}0 (10)^j0 (10)^k)
	% \end{align*}
	We first consider the cases with $\delta$-value three.
	\begin{theorem}\label{breakTwoEven}
		Let $x = (10)^c0 (10)^i0 (10)^j0 (10)^k$ be a string with three 100s and $|x|= 2n + 3$.
		If $c > 0$, or $i > 1$ and $k > 0$, then $x \notin R$.
	\end{theorem}
	\begin{proof}
		As stated in Theorem \ref{three100s}, the $G$-value of $100(10)^{n}$ is greater than $x$. Therefore, if $c > 0$, then $x \notin R$.

		For $c = 1$, we can write $G(x)$ as follows:
		\begin{equation} \label{breakingOdd}
			G(1\ 00 (10)^i0 (10)^j0 (10)^k ) = G((10)^i0 (10)^j0 (10)^k)
			+ G((10)^{i+1}0 (10)^j0 (10)^k),
		\end{equation}
		which is the sum of $G$-values of two strings with even length and two 100s.

		Now, in case $i>1$ and $k>0$, according to Theorem \ref{evenMult}, by decreasing $i$ by one and making $k = 0$, we obtain a greater $G$-value. Hence, the theorem holds.
	\end{proof}
	
	\begin{theorem} \label{begin100100}
		$100100\ (10)^n0$ and $100100\ (10)^{n-1}0\ 10$ are the only record-setters beginning with $100100$.
	\end{theorem}
	\begin{proof}
		The first record-setter begins with a 1000:
		\begin{equation*}
			G(1000\ 10\ (10)^n0) = 9F_{2n + 2} + 5F_{2n + 1}.
		\end{equation*}
		Now define $x_i = 100\ 100\ (10)^{i}0(10)^{n-i}$ for $1 \leq i \leq n$.  We get
		\begin{align*}
			G(x_i) &= G(100\ 100\ (10)^{i}0(10)^{n-i})\\
			&= 11(F_{2n + 1} + F_{2i}F_{2n - 2i + 1}) + 4(F_{2n} + F_{2i - 1}F_{2n - 2i + 1}) \\
			&= 11F_{2n + 1} + 4F_{2n} + F_{2n - 2i + 1}(11F_{2i} + 4F_{2i - 1})\\
			&= 11F_{2n + 1} + 4F_{2n} + F_{2n - 2i + 1}(4F_{2i + 2} + 3F_{2i}).
		\end{align*}
		As $i$ increases, the value of $G(x_i)$ also increases. Suppose  $i = n - 2$.  Then
		\begin{align*}
			G(x_{n - 2}) = G(100\ 100\ (10)^{n-2}0(10)^{2}) =
			11F_{2n + 1} + 4F_{2n} + 5(4F_{2n - 2} + 3F_{2n - 4}).
		\end{align*}
		This value is smaller than $9F_{2n + 2} + 5F_{2n + 1}$. Therefore, if $i < n -1$, then $x \notin R$.

		If $i = n - 1$, then for the string $x_{n - 1}$
		we have
		\begin{align*}
			G(x_{n - 1}) &= G(100\ 100\ (10)^{n-1}0(10))
			= 11F_{2n + 1} + 4F_{2n} + 2(4F_{2n} + 3F_{2n - 2})\\
			&= 11F_{2n + 1} + 12F_{2n} + 6F_{2n - 2} > 
			9F_{2n + 2} + 5F_{2n + 1}.
		\end{align*}
		Also, for $i = n$, we know $x_n > x_{n - 1}$. Therefore, the first two record-setters after $1000\ (10)^{n + 1}0$ are $100100\ (10)^{n - 1} 0 10$ followed by $100100\ (10)^n 0$.
	\end{proof}
	
	Putting this all together, we have the following result.
	\begin{theorem}
		The record-setters of odd length $2n + 3$, for $n \geq 5$,
		%??? for $n >=$ what ???
		are:
		
		$$\begin{cases}
			1000\ (10)^{n - 1}0,\\
			100\ 100\ (10)^{n - 3}0\ 10,\\
			100\ 100\ (10)^{n - 2}0,\\
			100\ (10)^{i}0\ (10)^{n - i - 1}0, &\text{ for } 1 < i \leq \lceil\frac{n-1}{2}\rceil, \\
			(10)^{i+1}0 (10)^{n-i},  & \text{ for } 0 \leq i \leq n.
		\end{cases}$$
		\label{oddthm}
	\end{theorem}
	\begin{proof}
	   % proven in ???
		The first three strings were already proven in Theorems~\ref{begin100100} and~\ref{odd1000}.
		
		We showed in Eq.~\eqref{breakingOdd} how to break the strings beginning with a 100 into two strings of even lengths. Thus, using Lemmas \ref{two100s} and \ref{evenMult}, for the strings of the form
		$$ 100 (10)^{i}0 (10)^{n - i - 1}0 $$
		for 
	$1 < i \leq \lceil\frac{n-1}{2}\rceil$,
	 the $G$-value increases with increasing $i$.
	
		Moreover,   Theorem~\ref{three100s} shows that the minimum $G$-value for a string having a single 100 is greater than the maximum with three 100s. So after the strings with three 100s come those with a single 100. Also, due to the Lemma \ref{oddFibZero} while using the same calculations as in Lemma \ref{minValSingle100}, as $i$ increases, we get greater $G$-values until we reach the maximum $F_{2n + 3} = G((10)^{n + 1})$.
	\end{proof}
	
	We can now prove
	Theorem~\ref{mainTheorem}.
	
	\begin{proof}[Proof of Theorem~\ref{mainTheorem}]
	By combining the results of
	Theorems~\ref{eventhm} and
	\ref{oddthm}, and noting that the indices for the sequence $s$ differ by $1$ from the sequence $a$, 
	the result now follows.
	\end{proof}
	
	We can obtain two useful corollaries of the main result.  The first gives an explicit description of the record-setters, and their $a$-values.
	
	\begin{corollary}
	The record setters lying in the interval $[2^{k-1}, 2^k)$ for $k \geq 12$ and even
	are, in increasing order
	\begin{itemize}
	\item $2^{2n-1} + {{2^{2n-2} - 2^{2n-2a-3}+1} \over {3}}$ for $0 \leq a \leq n-3$; 
	
	\item 
	${{2^{2n+1} - 2^{2n-2b} - 1} \over {3}}$ for $1 \leq b \leq \lfloor n/2 \rfloor$;
	and
	
	\item 
	$(2^{2n+1} + 1)/3$,
	\end{itemize}
	where $k = 2n$.
	
	The Stern values of these are, respectively
	\begin{itemize}
	    \item $L_{2a+3}F_{2n-2a-3} + L_{2a+1} F_{2n-2a-4}$
	    
	    \item $F_{2b+2} F_{2n-2b} + F_{2b} F_{2n-2b-1}$
	    
	    \item $F_{2n+1}$
	\end{itemize}
	where $L_0 = 2$, $L_1 = 1$, and $L_n = L_{n-1} + L_{n-2}$ for $n \geq 2$ are the Lucas numbers.
	
	The record-setters for $k\geq 12$ and odd are, in increasing order,
	\begin{itemize}
	    \item $2^{2n} + {{2^{2n-2} - 1} \over 3} $
	    \item $2^{2n} + 2^{2n-3} + {{2^{2n-4} -7} \over 3} $
	    \item $2^{2n} + {{2^{2n-1} -  2^{2n-2b-2} -1} \over 3}$ for $1 \leq b \leq \lceil n/2 \rceil - 1$
	    \item ${{2^{2n+2} - 2^{2n-2a-1} +1} \over 3}$ for $0 \leq a \leq n-2$;
	    \item $(2^{2n+2} - 1)/3$,
	\end{itemize}
	where $k = 2n+1$.
	
	The corresponding Stern values are
	\begin{itemize}
	    \item $F_{2n+1} + F_{2n-4}$
	    \item $F_{2n+1} + 8 F_{2n-8}$
	    \item $L_{2b+3} F_{2n-2b-2} + L_{2b+1} F_{2n-2b-3}$
	    \item $F_{2a+4} F_{2n-2a-1} + F_{2a+2} F_{2n-2a-2}$
	    \item $F_{2n+2}$
	    	\end{itemize}
	\end{corollary}
	
	\begin{proof}
	We obtain the record-setters from Theorem~\ref{mainTheorem} and the identities
	$[(10)^i]_2 = (2^{2i+1}-2)/3$
	and $[(10)^i 1]_2 = 
	(2^{2i+2}-1)/3$.
	
	We obtain their Stern values from Eq.~\ref{mat10}.
	
	\end{proof}
	
	\begin{corollary}
		The binary representations of the record-breakers for the
		Stern sequence form a context-free language.
	\end{corollary}
    
    \section{Acknowledgments}
    
    We thank Colin Defant for conversations about the problem in 2018, and for his suggestions about the manuscript.
  
    \bibliographystyle{new2}
	\bibliography{abbrevs,stern}

\newcommand{\noopsort}[1]{} \newcommand{\singleletter}[1]{#1}
\begin{thebibliography}{10}

\bibitem{Allouche&Shallit:1992}
J.-P. Allouche and J.~O. Shallit.
\newblock The ring of $k$-regular sequences.
\newblock {\em Theoret. Comput. Sci.} {\bf 98} (1992), 163--197.

\bibitem{Carlitz:1964}
L.~Carlitz.
\newblock A problem in partitions related to the {Stirling} numbers.
\newblock {\em Bull. Amer. Math. Soc.} {\bf 70} (1964), 275--278.

\bibitem{Coons&Tyler:2014}
M.~Coons and J.~Tyler.
\newblock The maximal order of {Stern's} diatomic sequence.
\newblock {\em Moscow J. Combin. Number Theory} {\bf 4} (2014), 3--14.

\bibitem{Defant:2016}
C.~Defant.
\newblock Upper bounds for {Stern's} diatomic sequence and related sequences.
\newblock {\em Electronic J. Combinatorics} {\bf 23} (2016), \#P4.8
  (electronic).

\bibitem{Lansing:2014}
J.~Lansing.
\newblock Largest values for the {Stern} sequence.
\newblock {\em J. Integer Sequences} {\bf 17} (2014), Article 14.7.5
  (electronic).

\bibitem{Lehmer:1929}
D.~H. Lehmer.
\newblock On {Stern's} diatomic series.
\newblock {\em Amer. Math. Monthly} {\bf 36} (1929), 59--67.

\bibitem{Lind:1969}
D.~A. Lind.
\newblock An extension of {Stern's} diatomic series.
\newblock {\em Duke Math. J.} {\bf 36} (1969), 55--60.

\bibitem{Lucas:1878}
E.~Lucas.
\newblock Sur les suites de {Farey}.
\newblock {\em Bull. Soc. Math. France} {\bf 6} (1878), 118--119.

\bibitem{Northshield:2010}
S.~Northshield.
\newblock Stern's diatomic sequence $0,1,1,2,1,3,2,3,1,4,\ldots$.
\newblock {\em Amer. Math. Monthly} {\bf 117} (2010), 581--598.

\bibitem{Paulin:2017}
R.~Paulin.
\newblock Largest values of the {Stern} sequence, alternating binary expansions
  and continuants.
\newblock {\em J. Integer Sequences} {\bf 20} (2017), Article 17.2.8
  (electronic).

\bibitem{Sloane:2022}
N.~J.~A. Sloane et~al.
\newblock The on-line encyclopedia of integer sequences.
\newblock Electronic resource, available at \url{https://oeis.org}, 2022.

\bibitem{Stern:1858}
M.~A. Stern.
\newblock {\"Uber} eine zahlentheoretische {Funktion}.
\newblock {\em J. Reine Angew. Math.} {\bf 55} (1858), 193--220.

\bibitem{Urbiha:2001}
I.~Urbiha.
\newblock Some properties of a function studied by {De Rham, Carlitz and
  Dijkstra} and its relation to the {(Eisenstein–)Stern's} diatomic sequence.
\newblock {\em Math. Commun.} {\bf 6} (2001), 181--198.

\end{thebibliography}
	% \begin{thebibliography}{9}
		
	% 	\bibitem{something}
	% 	$a(n) = (2F(2n+1) - Fibonacci(2n) - 2(-1)^n)/5$. - Ralf Stephan, May 14 2004
	% 	\bibitem{Fibs}
	% 	https://r-knott.surrey.ac.uk/Fibonacci/FibFormulae.html\#section5
	% 	\bibitem{f2}
	% 	https://oeis.org/A007598
	% \end{thebibliography}
	
\end{document}